%% file: stolovitch-gevrey5.tex
\documentclass[11pt]{article}
\usepackage[psamsfonts]{amsfonts}
\usepackage{cmmib57}
\usepackage{amsmath,amsthm,amssymb}
\usepackage[english]{babel}

\input{macro_L1.tex}

\input{macro_L2.tex}

\input{macro_E.tex}

\begin{document}

\title{Smooth Gevrey normal forms of vector fields near a fixed point}
\author{Laurent Stolovitch \thanks{CNRS, Laboratoire J.-A. Dieudonn\'e
U.M.R. 6621, Universit\'e de Nice - Sophia Antipolis, Parc Valrose
06108 Nice Cedex 02, France. courriel : {\tt stolo@unice.fr}. Ce travail a b\'en\'efici\'e d'une aide de l'Agence Nationale de la Recherche portant la r\'ef\'erence ``ANR-10-BLAN 0102''} 
 \ }

\date{\today}
\maketitle
\def\abstractname{Abstract} 
\begin{abstract}
We study germs of smooth vector fields in a neighborhood of a fixed point having an hyperbolic linear part at this point. It is well known that the ``small divisors'' are invisible either for the smooth linearization or normal form problem. We prove that this is completely different in the smooth Gevrey category. We prove that a germ of smooth $\al$-Gevrey vector field with an hyperbolic linear part admits a smooth $\be$-Gevrey transformation to a smooth $\be$-Gevrey normal form. The Gevrey order $\be$ depends on the rate of accumulation to $0$ of the small divisors. We show that a formally linearizable Gevrey smooth germ with the linear part satisfies Brjuno's small divisors condition can be linearized in the same Gevrey class.
\end{abstract}
\begin{center}
{\Large Formes normales Gevrey lisses de champs de vecteurs au voisinage d'un point fixe}
\end{center}
\def\abstractname{R\'esum\'e} 
\begin{abstract}
Nous \'etudions des germes lisses (i.e. $C^{\infty}$) de champs de vecteurs au voisinage d'un point fixe en lequel la partie lin\'eaire est hyperbolique. Il est bien connu que les petits diviseurs sont ``invisibles'' dans les probl\`emes de lin\'earisation ou de mise sous forme normale lisses. Nous montrons qu'il en est tout autrement dans la cat\'egorie Gevrey lisse. Nous montrons qu'un germe de champ de vecteurs $\al$-Gevrey lisse ayant une partie lin\'eaire hyperbolique au point fixe admet une transformation $\be$-Gevrey lisse vers une forme normale $\be$-Gevrey lisse o\`u l'indice $\be$ d\'epend de la vitesse d'accumulation vers z\'ero des ``petits diviseurs''. De plus, si le germe de champ de vecteurs, formellement linéarisable, est Gevrey lisse et admet une partie linéaire vérifiant la condition dioophantienne de Brjuno alors il est linéarisable dans la même classe Gevrey.
\end{abstract}
Mots-cl\'es : Syst\`emes dynamique hyperbolique, formes normales, lin\'earisation, petits diviseurs, r\'esonances, classes Gevrey.
\section{Introduction}

This article is concerned with the local behavior of solutions of germs of vector fields in a neighborhood of a fixed point in $\Bbb R^n$. More precisely, we are interested in the problem of classification under the action of the group of germs of diffeomorphisms preserving the fixed point (which will be chosen to be $0$ once for all). If $X$ is a germ of smooth (i.e. infinitely continuously differentiable) vector field and $\phi$ is a germ of diffeomorphism at the origin, then the action of $\phi$ on $X$ is defined to be : $\phi_*X:=D\phi(\phi^{-1})X(\phi^{-1})$. The space of formal diffeomorphisms also acts on the space of formal vector fields by the same formula. Since Poincar\'e, one of the main goal is to transform such a vector field to a new one, a model, which is supposed to be easier to study. Then, one then expects to obtain geometric and dynamical informations on the model and then pull them back to the original problem using the inverse transformation. So, the more regular (that is $C^k$, smooth, analytic,  ...) the transformation is, the more faithful the information will be.\\

Let us consider a (nonzero) linear vector field $L=\sum_{i=1}^n\left(\sum_{j=1}^n a_{i,j}x_j\right)\frac{\partial}{\partial x_i}$ of $\Bbb R^n$. We will consider smooth nonlinear perturbations of $L$, that is germs of vector fields of the form $X=L+R$ at the origin, where $R$ is a smooth vector field vanishing at $0$ as well as its derivative.
Such a vector field will be said formally linearizable if there exists a formal change of coordinates $x_i=y_i+\hat{\phi}_i(y)$, $i=1,\ldots, n$. 
such that the system of differential equations $\frac{d x_i}{dt}=\sum_{j=1}^n a_{i,j}x_j+R_i(x), i=1,\ldots, n$ can be written as $\frac{d y_i}{dt}=\sum_{j=1}^n a_{i,j}y_j, i=1,\ldots, n$.  It is well known \cite{Arn2} that, if $L=\sum_{i=1}^n\la_ix_i \frac{\partial}{\partial x_i}$ and if there is no multiindex $Q=(q_1,\ldots,q_n)\in \Bbb N^n$ with $|Q|:=q_1+\cdots+q_n\geq 2$ such that $\sum_{j=1}q_j\la_j=\la_i$ for some $1\leq i\leq n$, then any nonlinear perturbation of $S$ is formally linearizable. We say, in this situation, that there is no {\it resonances}. What about the regularity of the transformation ? Is it possible to find a convergent or smooth linearization ? One of the first and main result in this problem is the following theorem~:
\begin{theo}[Sternberg linearization theorem]\cite{stern2}\label{sternberg}
 Assume that $L$ is {\bf hyperbolic} (i.e. the eigenvalues of the matrix $(a_{i,j})$ have nonzero real parts). Let $X=L+R$ be a germ of smooth nonlinear perturbation of $L$. If $X$ is formally linearizable then $X$ is also smoothly linearizable; that is, there exists a germ of smooth diffeomorphism fixing $0$ which conjugate $X$ to $L$.
\end{theo}
We also refer to \cite{chaperon-ast,chaperon-ihes} for results and proofs in these circle of problems.

In the analytic category, the answer to the problem is not that simple~: it involves the {\bf small divisors} problem; that is, the rate of accumulation to $0$ of the nonzero numbers $\sum_{j=1}q_j\la_j-\la_i$ where the $\la_i$'s are the eigenvalues of the linear part. Let us define Bruno Diophantine condition~:
$$
(\omega)\quad -\sum_{k\geq 0}{\frac{\ln \omega_k}{2^k}} <+\infty
$$
where $\omega_k=\inf\{|(Q,\lambda)-\lambda_i|\neq 0, \;1\leq i\leq n,\;Q\in \mathbb N^n, 2\leq |Q|\leq 2^k\}$. One of the main result in the analytic category is the following~:
\begin{theo}[Bruno-Siegel linearization theorem]\cite{bruno}\label{bruno}
Assume that $S=\sum_{i=1}^n\la_ix_i \frac{\partial}{\partial x_i}$ satisfies Bruno condition $(\om)$  Let $X=S+R$ be a germ of nonlinear analytic perturbation of $S$ in a neighborhood of $0$. If $X$ is formally linearizable, then it is also analytically linearizable. That is, there is a germ of analytic diffeomorphism fixing the origin and which conjugate $X$ to $S$.
\end{theo}
C.L. Siegel \cite{siegel} was the first to prove such a statement under a stronger Diophantine condition, namely {\bf Siegel condition of order $\tau\geq 0$}~: there exists $c>0$ such that, for all $Q\in \Bbb N^n$ with $|Q|\geq 2$ for which $0\neq |(Q,\lambda)-\la_i|$, then
\begin{equation}\label{siegel}
|(Q,\lambda)-\la_i|\geq \frac{c}{|Q|^{\tau}}.
\end{equation} 
Although we don't know if condition $(\om)$ is necessary in order to have analycity of the linearizing transformation in dimension greater than $2$ (in dimension 2, this result is a consequence of Yoccoz theorem), it is easy to construct counter-example to convergence when the eigenvalues strongly violate Siegel condition $(\ref{siegel})$.

{\bf Theorem \ref{sternberg} shows that small divisors are invisible in the smooth category}. One of the main result of this article is to show that is not the same in the smooth Gevrey category (see definition \ref{gevrey-def} in appendix \ref{gevrey}). We first consider a germ of smooth $\al$-Gevrey non-linear perturbation of an hyperbolic linear part which is formally linearizable. {\bf We shall show that the small divisors affect the Gevrey character of the data : not only we show that there exists a germ of smooth Gevrey linearization at the origin but its Gevrey order depends on the behavior of the small divisors}. 
Let us first give few definitions.
\subsection{Definitions}
\begin{defi}
Let $\Om$ be an open set $\Bbb R^n$ and $\al\geq 1$. A smooth complex-valued function $f$ on an open set $\Om$ of $\Bbb R^n$ is said to be {\bf $\al$-Gevrey} if for any compact set $K\subset \Om$, there exist constants $M$ and $C$ such that, for all $k\in\Bbb N^n$,
$$
\sup_{x\in K}|D^{k}f(x)|\leq MC^{|k|}|k|!^{\al}.
$$
\end{defi}
\begin{defi}
A formal power series $\hat f=\sum_{Q\in \Bbb Q^n} f_Qx^Q$ is said to be $\be$-Gevrey if there exists positive constants $M,C$ such that $|f_Q|\leq MC^{|Q|}(|Q|!)^{\be}$.
\end{defi}
\begin{rem}
The Taylor expansion at the origin of a smooth $\al$-Gevrey function is a $(\al-1)$-Gevrey power series.
\end{rem}
\begin{defi}
A germ of smooth function at a compact set ${\cal K}$ is said to be {\bf flat on ${\cal K}$} is all its derivative vanish on $\cK$.
\end{defi}
\subsection{Statements}
Our first main result is~:
 \begin{theo}\label{gevrey-lin}
Let $\beta\geq \al\geq 1$. Assume that the linear part $S$ is hyperbolic and satisfies to the following condition~:
\begin{equation}\label{condition-marmi}
(\om_{\be,\al}):\quad \limsup_{k\rightarrow +\infty}\left(-2\sum_{p=0}^k\frac{\ln\om_{p+1}}{2^p}-\frac{1}{2^{k}}\ln (2^k!)^{\beta-\al}\right)<+\infty.
\end{equation}
Let $X=S+R$ be a germ of smooth $\al$-Gevrey nonlinear perturbation of $S$ in a neighborhood of the origin of $\Bbb R^n$. If $X$ is formally linearizable, then there exists a germ of smooth $\beta$-Gevrey diffeomorphism linearizing $X$ at the origin.
\end{theo}
\begin{coro}
If Brjuno condition $(\om)$ is satisfied, then any hyperbolic smooth $\al$-Gevrey vector field which is formally linearizable, is smoothly $\al$-Gevrey linearizable.
\end{coro}
\begin{coro}
If Brjuno series is divergent as $\frac{1}{2^{k}}\ln (2^k!)^{\beta-\al}$, then any hyperbolic smooth $\al$-Gevrey vector field which is formally linearizable, is smoothly $\be$-Gevrey linearizable.
\end{coro}
What happens when $X$ is not formally linearizable ? It is possible, via a formal diffeomorphism, to transform $X$ into a ({\it a priori} formal) model, called a {\bf normal form}.
\begin{defi}\label{def-nf}
 Let $S=\sum_{i=1}^n\la_ix_i\frac{\partial}{\partial x_i}$ be a diagonal linear vector field of $\Bbb R^n$. A formal vector field $U$ is a {\bf normal form} with respect to $S$ if it commutes with $S$~: $[S,U]=0$, where $[.,.]$ denotes the Lie bracket of vector fields.
\end{defi}
An analytic perturbation of $S$ does not have, in general, an analytic transformation to a normal form. In fact, besides the small divisors condition, one also needs to impose some algebraic conditions on the normal form (``complete integrability condition'') in order to obtain the holomorphy of a normalizing transformation. These phenomena have been studied in \cite{bruno,Stolo-ihes,Stolo-cartan,vey-iso,vey-ham,Ito1,zung-birkhoff}. For a recent survey and lecture note, we refer to \cite{Stolo-nonlin,Stolo-asi07}. One of the main problem was then to quantify this generic divergence : how far from convergence a formal transformation to a normal form can be ? The fundamental problem was solved recently by G. Iooss and E. Lombardi \cite{IoossLombardi} and then generalized by E. Lombardi et L. Stolovitch \cite{stolo-lombardi}. They proved that a nonlinear analytic perturbation of a linear vector field satisfying Siegel condition admits a formal Gevrey transformation to a formal Gevrey normal form. The Gevrey order depends on the rate of accumulation to zero of the small divisors. Since this is just at the formal level, this is not suitable to get dynamical nor geometrical information. So we wanted to know whether we could find a genuine smooth transformation having these Gevrey properties. This is the goal our second main result deals with the conjugacy problem to a normal form in the Gevrey category~:
\begin{theo}\label{gevrey-nf}
Let $\al\geq 1$. Assume that $S$ is an hyperbolic linear diagonal vector field satisfying to Siegel condition of order $\tau$. Let $X=S+R$ be a smooth $\al$-Gevrey nonlinear perturbation of $S$, then there exists a germ of smooth $(\al+\tau+1)$-Gevrey conjugacy of $X$ to a germ of smooth $(\al+\tau+1)$-Gevrey normal form at the origin. 
\end{theo}
The following presentation as been sugested by M. Zhitomirskii~\footnote{added in proof}:
\begin{coro}\label{zhito}
Let $X$ be a $\al$-Gevrey vector field as in the previous theorem and let $\be:=\al+\tau+1$. Then, for any smooth $\be$-Gevrey flat vector field $Z$ at the origin, $X+Z$ is smoothly $\be$-Gevrey conjugate to $X$. 
\end{coro}
\begin{proof}
According to the previous theorem, there exists a smooth $\be$-Gevrey diffeomorphism $\Phi$ that conjugates $X$ to a smooth $\be$-Gevrey normal form $N$. Moreover, $\ti Z:=\Phi_*Z$ is a smooth $\be$-Gevrey flat vector field in a neighborhood of the origin. Hence, according to $(\ref{probleme-plat})$, there exists a smooth $\be$-Gevrey diffeomorphism $\Psi$ that conjugates $N+\ti Z$ back to $N$. As a consequence, $\Psi\circ\Phi_*(X+Z)=\Phi_*X$.
\end{proof}
In fact, we will prove a stronger version of these two theorems since we will not assume the linear part to be ``diagonal''.\\ 

In these two results, one sees the {\bf impact of the small divisors on the Gevrey order of the conjugacy} to a normal form (which is the linear part in the first case). 

This connection between small divisors and Gevrey character appeared also in the context of holomorphic saddle-nodes \cite{stolo-boele}. In connection with KAM problem, G. Popov already considered smooth Gevrey normal form \cite{popov-ihp1,popov-herman}. He constructs smooth conjugacy of Hamiltonians to smooth Gevrey normal form {\bf up to an exponentially small remainder}. One of the main goal of this article is to show that, under the hyperbolicity condition, one can get rid of this remainder, thus obtaining a genuine smooth Gevrey conjugacy to a genuine smooth Gevrey normal form.

We have gathered in appendix \ref{gevrey}, all the results about Gevrey functions and formal power series we use.\\
\subsection{Idea of the proof}
Let us give a sketch of the proof. The first ingredient are recent results about the existence of a formal Gevrey transformation to a formal Gevrey normal form. For the first theorem, we shall use the theorems by Marmi-Carletti\cite{carletti-marmi} and by Carletti\cite{carletti} that says that there is a formal $(\be-1)$-Gevrey linearization. For the second theorem, we use the results by Iooss-Lombardi\cite{IoossLombardi} and by Lombardi-Stolovitch\cite{stolo-lombardi} that says there exists a formal $(\al+\tau)$-Gevrey transformation to a formal $(\al+\tau)$-Gevrey formal form $\hat N$. Our second ingredient is a Gevrey version of the Whitney extension theorem due to Bruna~: we can realize these formal Gevrey transformation and normal form as the Taylor expansion at the origin of germs of smooth Gevrey objects at the origin. The realization of the formal normal form as the Taylor expansion of a {\bf smooth normal form} is not just an application of Bruna theorem. Indeed, we have by definition $[S,\hat N]=0$. If $Y$ is any germ of smooth $(\al+\tau+1)$-Gevrey vector field at the origin ``realizing'' $\hat N$, there is no reason to have also $[S,Y]=0$. So, in order to find a smooth normal form realizing $\hat N$, we have first to solve the cohomological equation $[S,U]=-[S,Y]$ where $U$ is the unknown and $[S,Y]$ is flat. To do so, we already have to apply our main ``sub-theorem'' (theorem \ref{main-0}) that solves the cohomological equation with flat Gevrey data right hand side. Then, $N:=Y+U$ will be a smooth $(\al+\tau+1)$-Gevrey normal form realizing $\hat N$. In that case, we set $\beta:=\al+\tau+1$.

We shall then show that there exists a germ of smooth $\be$-Gevrey diffeomorphism $\phi$ (resp. normal form $N$ which is $S$ in the first case) at the origin such that $R:=\phi_*X-N$ is a germ of smooth flat $\beta$-Gevrey vector field at the origin. The main problem now is to show that we can get rid of the flat remainder by the mean of a germ of $\be$-Gevrey smooth diffeomorphism $\psi$ such that $\psi-\Id$ is flat at the origin~: $\psi_*(N+R)=N$. This will be solved also by theorem \ref{main-0}. Since the composition of two $\be$-Gevrey maps is also a $\be$-Gevrey map, we will obtain a smooth $\be$-Gevrey conjugacy to a normal form $(\psi\circ\phi)_*X=N$. In order to prove theorem \ref{theo-flat} (theorem \ref{main-0} is special case of it), we shall follow and improve the estimates of the proof that R. Roussarie gives for this problem in the smooth case \cite{roussarie-ast}. We shall also give a Gevrey version of the stable and unstable manifold theorem.

\section{Formal Gevrey conjugacy}

Let us recall some recent results about formal normalization. 

First of all, we recall some facts from \cite{stolo-lombardi,stolo-lombardi-cras}. Let us define a scalar product on the space of polynomials as follow~: $<x^Q,x^P>=\frac{Q!}{|Q|!}$ if $P=Q\in \mathbb N^n$ and $0$ otherwise \cite{IoossLombardi,stolo-lombardi, fischer}. As usual, $Q!=q_1!\cdots q_n!$ and $x^Q=x_1^{q_1}\cdots x_n^{q_n}$. It is known that a formal power series $\sum_kf_k$ where $f_k$ is an homogeneous polynomial of degree $k$ defines a germ of analytic function at the origin if and only if there exists a $c>0$ such that, for all $k\ \mathbb N$, $|f_k|\leq c^k$ \cite{shapiro} where $|f_k|$ denotes the norm with respect to the scalar product. This means that , if $f_k=\sum_{Q\in \Bbb N^n,|Q|=k} f_Qx^Q$, then $|f_k|^2=\sum_{Q\in \Bbb N^n,|Q|=k} |f_Q|^2\frac{Q!}{k!}$. The induced scalar product on the space of polynomial vector fields is defined as ~: $<X,Y>=\sum_{i=1}^n<X_i,Y_i>$ where we have written $X=\sum_{i=1}^nX_i\frac{\partial}{\partial x_i}$.
Let $L=\sum_{i=1}^n\left(\sum_{j=1}^n a_{i,j}x_j\right)\frac{\partial}{\partial x_i}$ be a linear vector field of $\Bbb R^n$. Let $\cH_k$ be the space of homogeneous vector fields of degree $k$. Let $d_0:\cH_k\rightarrow \cH_k$ be the linear operator $d_0(U):=[L,U]$ where $[.,.]$ denotes the Lie bracket of vector fields. We define the box operator $\square_k:=d_0d_0^t$ where $d_0^t$ denotes the transpose of $d_0$. It is known \cite{belitskii-nf,iooss-elphick} that $d_{0|\cH_k}^t=\frac{1}{k!}[L^t,.]$ where $L^t:=\sum_{i=1}^n\left(\sum_{j=1}^n a_{j,i}x_j\right)\frac{\partial}{\partial x_i}$. Let us define $a_{k}:=\min\sqrt{\la}$ the minimum is taken over the set $\sig_k$ of nonzero eigenvalues $\la$ of $\square_k$. We shall say that $L$ satisfies {\bf Siegel condition of order $\tau$} if there exists a constant $c$ such that, for all $k\geq 2$, 
\beq \label{siegel-gen}
a_k\geq \frac{c}{k^{\tau}}.
\eeq
\begin{rem}
If $L=\sum_{i=1}^n\la_ix_i\frac{\partial }{\partial x_i}$, then $a_k$ is the minimum of the nonzero $|(Q,\la)-\la_i|$'s for all $Q\in \Bbb N^n$ with $|Q|=k+1$ and $1\leq i\leq n$.
\end{rem}
\begin{theo}\cite{IoossLombardi}
If $X=L+R$ is a nonlinear analytic perturbation of $L$ and if $L$ satisfies Siegel condition of order $\tau$, then there exists a $(1+\tau)$-Gevrey formal conjugacy of $X$ to a formal $(1+\tau)$-Gevrey normal form. 
\end{theo}
This result does appear under this form in the aforementioned article. Although, it has been generalized to perturbations of quasi-homogeneous vector fields \cite{stolo-lombardi,stolo-lombardi-cras}, we only use our version for perturbation of linear vector fields.
\begin{theo}\cite{stolo-lombardi,stolo-lombardi-cras}\label{formal-nf}
If $X=L+R$ is a nonlinear smooth $\al$-Gevrey perturbation of $L$ and if $L$ satisfies Siegel condition of order $\tau$, then there exists a $(\al+\tau)$-Gevrey formal conjugacy of $X$ to a formal $(\al+\tau)$-Gevrey normal form. 
\end{theo}
The proof amount to find a formal diffeomorphism $\Id+\sum_{k\geq 1}U_k$ where $U_k$ is a homogeneous polynomial vector field of degree $k+1$ such that 
$$
|U_k|\leq Mk!^{\tau+\al}c_k
$$
Here $\sum c_kt^k$ is a formal solution of nonlinear differential equation with an irregular singularity at the origin which satisfies $c_k\leq Mk!$ (see \cite{stolo-lombardi}[theorem 6.4, remark 6.7]).\\

Carletti and Marmi in dimension $1$ \cite{carletti-marmi} and then Carletti in any dimension \cite{carletti}, have studied the problem of formal linearization of formal Gevrey vector fields and diffeomorphisms. The main result can be summarized as follow~:
\begin{theo}\cite{carletti}\label{carletti}
Let $\beta\geq \al\geq 0$. Assume that the linear part $S=\sum_{i=1}^n\la_ix_i \frac{\partial}{\partial x_i}$ satisfies to the following condition~:
\begin{equation}\label{condition-marmi}
(\om_{\be,\al}):\quad \limsup_{k\rightarrow +\infty}\left(-2\sum_{p=0}^k\frac{\ln\om_{p+1}}{2^p}-\frac{1}{2^{k}}\ln (2^k!)^{\beta-\al}\right)<+\infty.
\end{equation}
Let $\hat X=S+\hat R$ be a formal $\al$-Gevrey perturbation of $S$. If $\hat X$ is formally linearizable, then there is an $\beta$-Gevrey formal linearization.
\end{theo}
In fact, in the aforementioned article, the condition of ``non resonances'' is assumed in order to have formal linearization.
\begin{coro}
 \begin{itemize}
\item If the linear part $S$ satisfies to Bruno condition $(\om)$ and if the formal $\al$-Gevrey perturbation of $S$ is formally linearizable, then there exists a $\al$-Gevrey formal transformation to the linear part.
  \item Assume that the linear part $S$ satisfies to 
$$
\limsup_{k\rightarrow +\infty}\left(-2\sum_{p=0}^k\frac{\ln\om_{p+1}}{2^p}-\frac{1}{2^{k}}\ln (2^k!)^{\beta}\right)<+\infty.
$$
Let $X=S+ R$ be a non-linear analytic perturbation of $S$. If $X$ is formally linearizable, then there is a $\beta$-Gevrey formal linearization.
 \end{itemize}
\end{coro}
\begin{proof}
 In the first case, apply theorem \ref{carletti} with $\al=\beta$. In the last case, set $\al=0$.
\end{proof}

The proof when $\al=0$ amount to find a formal diffeomorphism $\Id+\sum_{k\geq 1}U_k$ such that 
$$
|U_{k}|\leq \eta_kc_k
$$
where $\sum c_kt^k$ is a formal solution of analytic implicit function problem (see \cite{stolo-lombardi}[lemma 5.9]), hence $c_k\leq mC^k$ . Here,
$\eta_k$ is the sequence of positive numbers defined as follow~: $\eta_0=1$, and for $k>0$, 
$$
a_{k+1}\eta_{k}=\max_{1\leq\mu\leq k}\max_{{k_1+\cdots+k_{\mu+1}+\mu=k}} \eta_{k_1}\cdots\eta_{k_{\mu+1}}.
$$

In order to obtain the convergence, it is sufficient to assume that, for all positive integer $k$,  $\eta_k\leq c^k$ for some positive constant $c$. If $L$ is linear diagonal, then Bruno condition $(\om)$ precisely implies that this holds \cite{bruno,stolo-dulac}.

In order to obtain the $\al$-Gevrey version with a general linear part $L$, if is sufficient to consider instead the sequence $\eta_k$ defined as follow~: $\eta_0=1$, and for $k> 0$, 
$$
a_{k+1}\eta_{k}:=\max_{1\leq\mu\leq k}\max_{{k_1+\cdots+k_{\mu+1}+\mu=k}} ((\mu+1)!)^{\al}\eta_{\de_1}\cdots\eta_{\de_{\mu+1}}.
$$
\begin{prop}\label{general-lin}
Assume that there exists $C>0$ such that for all $k\in\mathbb N^*$, 
\beq\label{gen-lin} \eta_k\leq C^k(k!)^{\be}.\eeq Then, if the formal $\al$-Gevrey perturbation of $L$ is formally linearizable, then there is a formal $\be$-Gevrey linearization.
\end{prop}
\begin{proof}
We just show how to adapt proof of \cite{stolo-lombardi}[theorem 5.8] to Gevrey data. We refer to this article for the full details.
Let $X=L+R$ be the nonlinear perturbation of $L$. Let $\phi^{-1}=\Id+U=\Id+\sum_{k\geq 1}U_k$ be a linearizing transformation. We have $[L,U]=R(I+U)$. Thus, if we decompose into homogeneous components, we obtain the following estimate
$$
\forall \de\geq 1,\quad \left(\min_{\lambda\in \sigma_{\de+1}}\sqrt \lambda\right)
|U_{\de}|\leq |\{R(Id+U)\}_{\de}|.
$$
where $\{R(Id+U)\}_{\de}$ denotes the homogeneous polynomial of degree $\de+1$ in the Taylor expansion at $0$. Let $R_{\mu}$ be the homogeneous polynomial of degree $\mu+1$ of the Taylor expansion of $R$ at the origin. We then denote by $\ti R_{\mu}$ the unique $\mu+1$-linear map such that $\ti R_{\mu}(x,\ldots, x)=R_{\mu}(x)$.
We have
\begin{eqnarray*}
\{R(Id+U)\}_{\de} & = & \left\{\suml_{\mu> 0}R_{\mu}(Id+U)\right\}_{\de}\\
& = & \left\{\suml_{\mu> 0}\tilde
R_{\mu}({\underbrace{Id+U,\ldots,Id+U}_{\text{$\mu+1$ times}}})\right\}_{\de}\\
& = & \suml_{\mu>
0}\suml_{\delta_1+\cdots+\delta_{\mu+1}+\mu=\de}\tilde
R_{\mu}(U_{\delta_1},\ldots,U_{\delta_{\mu+1}})
\end{eqnarray*}
where the $\delta_i$'s are nonnegative integers and where we have set $U_0:=Id$.
Hence, since the scalar product is sub-multiplicative \cite{stolo-lombardi}[proposition 3.6], then
$$
|\{R(Id+U)\}_{\de}|\leq \suml_{\mu>
1}\suml_{\delta_1+\cdots+\delta_{\mu+1}+\mu=\de}\|\tilde
R_{\mu}\||U_{\delta_1}|\cdots |U_{\delta_{\mu+1}}|.
$$
Since, $R$ is $\al$-Gevrey, then there exists $C>0$ such that $\|\tilde R_{\mu}\|\leq C^{\mu+1}((\mu+1)!)^{\al}$. Let us define the sequence $\{\sigma_{\de}\}_{\de\in \Bbb N}$ of positive numbers defined by $\sigma_{0}:=\|Id\|_{p,0}$ and if $\de$ is positive,
$$
\sigma_{\de}:=\suml_{\mu>0}^{\de}
              \suml_{\delta_1+\cdots+\delta_{\mu+1}+\mu=\de}
              C^{\mu+1}\sigma_{\delta_1}\cdots \sigma_{\delta_{\mu+1}}
$$
where the $\delta_i$'s are nonnegative integers. As in \cite{stolo-lombardi}[lemma 5.10], we can show that the series $\sum\sig_{\de}t^{\de}$ converges in a neighborhood of $0$. Moreover, we can show by induction as in \cite{stolo-lombardi}[lemma 5.9], that for all $\de\geq 1$, $|U_{\de}|\leq \eta_{\de}\sig_{\de}$. Hence, if $(\ref{gen-lin})$ is satisfied then $|U_{\de}|\leq D^{\de}(\de!)^{\be}$ for some positive $D$. This ends the proof.
\end{proof}
\begin{lemm}
If $L$ is a diagonal linear vector field that satisfies Bruno-Carletti-Marmi condition $(\om_{\be,\al})$, then it satisfies condition $(\ref{gen-lin})$.
\end{lemm}
\begin{proof}
In fact, according to the estimate of \cite{stolo-dulac}[p.1411] (or \cite{bruno}[p.216-222]), we have for $2^l+1\leq k\leq 2^{l+1}$,
$$
\eta_{k-1}\leq  (k!)^{\al}\prod_{j=0}^l\left(\frac{2}{\omega_{j+1}}\right)^{2n\tfrac{k}{2^l}}.
$$
Therefore, by taking the $\log$, we have
$$
\log\eta_{k-1}\leq k\left(\frac{\log{(k!)^{\al}}}{k}+\left(-2n\sum_{j= 0}^l{\frac{\ln \omega_{j+1}}{2^j}}+2n\ln 2\sum_{j\geq 0}{\frac{1}{2^j}}\right)\right).
$$
As a conclusion, if Carletti-Marmi condition $(\om_{\al,\be})$ holds, then for some constant $C$, we have 
$$
\log\eta_{k-1}\leq k\left(\frac{\log{(k!)^{\be}}}{k}+C\right)
$$
and we are done.
\end{proof}

Let us give an example showing that we can obtain the prescribed divergence. This example is adapted from the one constructed by J.-P. Fran\c{c}oise \cite{francoise-book} to show the defect of holomorphy of the linearization of an analytic perturbation of a linear vector field with Liouvillian eigenvalues. 

Let $\be\geq \al\geq 0$. Let us assume that the irrational number $\ze$ is {\it  Liouvillian} and that there exist two sequences of positive 
integers $(p_n),(q_n)$ both tending to infinity with $n$ such that
$$
\left|\ze -\frac{p_n}{q_n}\right|<\frac{1}{q_n(q_n!)^{\be-\al}}.
$$
for some $\be\geq 1$.
Then, let us consider the formal $\al$-Gevrey function (unit)
$$
f(x,y)=\frac{1}{1-\sum (q_n!)^{\al}x^{p_n}y^{q_n}}.
$$
Let us consider the linear vector field $S:=x\frac{\partial}{\partial x}-\ze y\frac{\partial}{\partial y}$.
Let us consider the following formal $\al$-Gevrey perturbation of $S$ : $X=f.S$. It is formally linearizable. It is shown in \cite{Stolo-asi07}[example 1.3.3] that the unique linearizing transformation $x'=x\exp(-V(x,y))$, $y'=y\exp(-W(x,y))$ is given by $V(x,y)=\sum \frac{(q_n!)^{\al}}{p_n-\ze q_n}x^{p_n}y^{q_n}$. It is a formal power series that diverges at least as a $\beta$-Gevrey series since $(q_n!)^{\be}<\frac{(q_n!)^{\al}}{p_n-\ze q_n}$. If furthermore we require that a lower bound 
\beq\label{dioph-gevrey}
\frac{c}{q_n(q_n!)^{\be-\al}}\leq \left|\ze -\frac{p_n}{q_n}\right|
\eeq 
is satisfied as well then the linearizing transformation is exactly $\be$-Gevrey. 



\subsection{Proof of the main theorems \ref{gevrey-lin} and \ref{gevrey-nf}}

Let $\be\geq\al\geq 1$ and let $X=S+R$ be a smooth $\al$-Gevrey nonlinear perturbation of $S$ in a neighborhood of the origin in $\Bbb R^n$. Its Taylor expansion at the origin is a formal $(\al-1)$-Gevrey power series. If Siegel condition $(\ref{siegel})$ of order $\tau$ is satisfied, then, according to theorem \ref{formal-nf}, there exists a formal $(\al+\tau)$-Gevrey diffeomorphism $\hat\Phi$ conjugating $X$ to a formal $(\al+\tau)$-Gevrey normal form $\hat N$. 
If condition $(\ref{gen-lin})$ is satisfied and if $X$ is formally linearizable then, according to proposition \ref{general-lin}, there exists a formal $(\be-1)$-diffeomorphism $\hat\Phi$ conjugating $X$ to the linear part $S$. In both cases, we have $\hat\Phi_*X=\hat N$. In the first case, we shall set $\be:=\al+\tau+1$. In the second case, we shall set $\hat N:=S$. Hence, we shall prove that there exists a germ of smooth $\be$-Gevrey diffeomorphism conjugating $X$ to a germ of smooth $\be$-Gevrey normal form (in the second case, this normal form is just $S$ itself). 

Both $\hat \Phi$ and $\hat N$ have components which are $(\be-1)$-Gevrey formal power series. According to a Gevrey-Whitney theorem \ref{bruna} applied to $K=\{0\}$, there exists a germ of smooth $\be$-Gevrey diffeomorphism $\Phi$ (resp. germ of smooth $\beta$-Gevrey vector field $Y$) which has $\hat\Phi$ as Taylor jet at the origin (resp. $\hat N$). In the linearization case, we can choose $Y$ to be $S$, which is, of course, $\be$-Gevrey. Hence, the Taylor jet at the origin of $\Phi_*X$ and that of $Y$ are equal. Since $\be\geq \al$, $X$ can be regarded as a germ of smooth $\be$-Gevrey vector field at the origin of $\Bbb R^n$. Let us first recall that $\Phi_*X(y)=D\Phi(\Phi^{-1}(y))X(\Phi^{-1}(y))$. Since $\Phi^{-1}$ is smooth $\be$-Gevrey map \cite{komatsu-inverse}, then $X\circ \Phi^{-1}$ is also smooth $\be$-Gevrey vector field (the composition of two smooth $\be$-Gevrey mappings is also $\be$-Gevrey smooth). Since the derivative of a smooth $\be$-Gevrey function is also $\be$-Gevrey smooth and the product of two smooth $\be$-Gevrey functions is also $\be$-Gevrey smooth, we proved that obtain that $\Phi_*X$ is a germ of smooth $\be$-Gevrey vector field. 

{\bf All $1$-parameter families of smooth functions or vector fields considered are supposed to be defined on a same neighborhood of the origin.}\\
We shall postpone the proof of the following theorem to the next section.
\begin{theo}\label{main-0}
Let $\{X_t\}_{t\in [0,1]}$ be a $1$-parameter family of germs of smooth $\be$-Gevrey vector fields at the origin of $\Bbb R^n$ vanishing of $0$. Assume that $\{X_t\}_{t\in [0,1]}$ is hyperbolic, uniformly in $t\in[0,1]$. Let $\{Y_t\}_{t\in [0,1]}$ be a 1-parameter family of germs of smooth $\be$-Gevrey flat vector fields at $0$. Then, the equations for all $0\leq t\leq 1$
\beq
[X_t,Z_t] =  Y_t\label{Lie-bra}
\eeq
have a 1-parameter of germs of smooth $\be$-Gevrey flat solutions $\{Z_t\}_{t\in [0,1]}$.
\end{theo}

Let us go back to the proof of our main theorems. The first problem we face is that $Y$ is not, {\it a priori}, a normal form. We will show that we can add to it a germ of smooth flat $\beta$-Gevrey vector field $U$ at the origin so that $N=Y+U$ is a smooth $\beta$-Gevrey normal form. In fact, since $\hat N$ is a normal form, we have $[S,\hat Y]=0$. Let us set $r:=[S,Y]$. It is a germ of smooth flat $\beta$-Gevrey vector field at the origin since derivation and products preserve the Gevrey character. If the linear part $S$ is hyperbolic, then theorem \ref{main-0} (applied with the constant family $X_t=S$, $Y_t=-r$) shows that there exists a smooth flat $\beta$-Gevrey vector field $U$ such that $[S,U]=-r$. Hence, $[S,Y+U]=0$ and $N:=Y+U$ is a smooth $\beta$-Gevrey vector field admitting $\hat N$ as Taylor expansion at the origin.

As a consequence, $R:=\Phi_*X-N$ is a germ of smooth $\be$-Gevrey vector field which Taylor jet at the origin vanishes.
We will then prove that, if the linear part is hyperbolic, there exists a germ of smooth $\be$-Gevrey diffeomorphism $\Psi$ at the origin of $\Bbb R^n$, infinitely tangent to the identity at this point, such that
\beq\label{probleme-plat}
\Psi_*(N+R)=N.
\eeq
We will adapt the proof of R. Roussarie \cite{roussarie-ast} who solves equation $(\ref{probleme-plat})$ in the smooth flat category. We apply the homotopic method ("la m\'ethode des chemins"), that is we are looking for a $1$-parameter family of germs of smooth $\be$-Gevrey diffeomorphisms $\{\Psi_t\}_{t\in [0,1]}$ fixing the origin such that for all $t\in [0,1]$, 
\beq\label{conj}
(\Psi_t)^{-1}_*N=N+tR.
\eeq
Let us first solve the (family of) cohomological equation
\beq\label{equ-cohom-param}
[N+tR, Z_t]= R
\eeq
where $Z_t$ is the unknown family of flat vector fields. Since the family $N+tR$ is hyperbolic at the origin and smooth $\be$-Gevrey then, according to theorem \ref{main-0}, there exists a family of germs of smooth $\be$-Gevrey vector fields $Z_t$ solution of $\ref{equ-cohom-param})$.
It is then well known (\cite{dumortier-book}[section 2.7]) that the solution obtained by solving equation $\ref{equ-cohom-param})$ will give the solution to $(\ref{conj})$. Indeed, $(\Psi_t)^{-1}$ is defined to be the solution of (see \cite{dumortier-book}[lemma 2.21])
\beq\label{edo-cohom}
\frac{d (\Psi_t)^{-1}}{dt} = Z_t((\Psi_t)^{-1}).
\eeq
Then, according to Komatsu theorem \ref{komatsu-edo}, the unique solution $\Psi_t^{-1}$ is a germ of smooth $\be$-Gevrey map uniformly in $t\in[0,1]$. Hence, $\Psi_t$ has the same property.

On the one hand, $\Psi_1\circ \Phi$ is a germ of smooth $\be$-Gevrey diffeomorphism at the origin. On the other hand, it satisfies to
$$
(\Psi_1\circ \Phi)_*X=(\Psi_1)_*(\Phi_*X)=(\Psi_1)*(N+R)=N.
$$
Hence, $\Psi_1\circ \Phi$ is a smooth $\be$-Gevrey conjugacy to a smooth $\be$-Gevrey normal form and we are done.

\section{Solution of cohomological equations with smooth Gevrey flat data}
The goal of the section is to prove theorem \ref{main-0}. In fact, we shall prove a more general result.

Let $F:=\Bbb R^p\times\{0\}\in  \Bbb R^p\times\Bbb R^{n-p}=\Bbb R^n$. If $x=(x_1,\ldots, x_n)$ denotes local coordinates, we write $x':=(x_1,\ldots, x_p)$ and $x'':=(x_{p+1},\ldots, x_n)$; so that $F=\{x_{p+1}=\cdots=x_n=0\}$. \\

{\bf All $1$-parameter families of smooth functions or vector fields considered are supposed to be defined on a same neighborhood of the origin.}\\

Let $1<\al$. Let ${\cal G}^{< -\al}(F)$ (resp. ${\cal G}^{< -\al}_t(F)$) be the ring of (resp. $1$-parameter families of) germs of smooth $\al$-Gevrey functions at th origin which are flat on $F$ (resp. uniformly in $t\in [0,1]$).
Let $G^{< -\al}_n(F)$, resp. $G^{< -\al}_{t,n}(F)$) be the vector space of (resp. $1$-parameter families of) germs of smooth $\al$-Gevrey vector fields at the origin which are flat on $F$ (resp. uniformly in $t\in [0,1]$).\\

\begin{defi}
Let $\{X_t\}_{t\in [0,1]}$ be a $1$-parameter family of germs of smooth vector fields vanishing on $F$, uniformly in $t\in [0,1]$. We shall say that $X_t$ is {\bf non-degenerate hyperbolic transversally} to $F$ if its linear part at the origin of $\Bbb R^n$ can be written as
$$
J^1(X_t)= \sum_{i,j=p+1}^na_{i,j}(t)x_i\frac{\partial }{\partial x_j}+\sum_{j=1}^p\sum_{i=p+1}^nb_{i,j}(t)x_i\frac{\partial }{\partial x_j}
$$
and if the eigenvalues $\lambda_{p+1}(t),\ldots,\lambda_n(t)$ of the matrix $(a_{i,j}(t))_{i,j=p+1,\ldots, n}$ have a nonzero real part.
\end{defi}
We shall prove the following result~:
\begin{theo}\label{theo-flat}
Let $\{X_t\}_{t\in [0,1]}$ be a $1$-parameter family of germs of smooth $\be$-Gevrey vector fields at the origin of $\Bbb R^n$ vanishing of $F$. Assume that $\{X_t\}_{t\in [0,1]}$ is non-degenerate hyperbolic transversally to $F$, uniformly in $t\in[0,1]$. If $h_t\in {\cal G}^{< -\be}_t(F)$ and $Y_t\in G^{< -\be}_{t,n}(F)$ then the equations
\beq
{\cal L}_{X_t}(f_t) = h_t \label{Lie-der}
\eeq
and 
\beq
[X_t,Z_t] =  Y_t\label{Lie-bra}
\eeq
have a solution $f_t\in {\cal G}^{< -\be}_t(F)$ and $Z_t\in G^{< -\be}_{t,n}(F)$ respectively.
\end{theo}
We follow the scheme of the proof that R. Roussarie has done in the context of smooth flat objects \cite{roussarie-ast}[chapitre 1, section 2, p.37-45]. Theorem \ref{main-0} corresponds to the case where $F=\{0\}$.
\subsection{Case of a contraction}
This section is devoted to prove the previous theorem in the case where we have a normal contraction to a subspace $F$. Namely, we prove the following
\begin{prop}\label{prop-roussarie}
Let $F:=\Bbb R^p\times\{0\}\in  \Bbb R^p\times\Bbb R^{n-p}$ as above. Let $\{X_t=X'_t+X''_t\}_{t\in[0,1]}$ be $1$-parameter family of germs of smooth $\be$-Gevrey vector fields at the origin of $\R^n$. We assume that
\begin{enumerate}
\item $X'_t=\sum_{i=p+1}^{n}X'_i(t,x)\frac{\partial }{\partial x_i}$ vanishes on $F=\{x_{p+1}=\cdots=x_{n}=0\}$, and $X'_t$ is normally contracting to $F$, that is the eigenvalues $\lambda_{p+1}(t),\ldots,\lambda_n(t)$ of its linear part at the origin have a  negative real part.
\item $X''_t=\sum_{i=1}^{p}X''_i(t,x)\frac{\partial }{\partial x_i}$ (it is not assumed to vanish on $F$).
\end{enumerate}
Then, equations $(\ref{Lie-der})$ and $(\ref{Lie-bra})$ have solutions $f_t\in {\cal G}^{< -\be}_t(F)$ and $Z_t\in G^{< -\be}_{t,n}(F)$  respectively, whenever $h_t\in {\cal G}^{< -\be}_t(F)$ and $Y_t\in G^{< -\be}_{t,n}(F)$.
\end{prop}
\begin{proof}
We can assume that $X_t$ is defined on $V=B_1\times B_2\subset \Bbb R^p\times \Bbb R^{n-p}$ for some open ball $B_1$(resp. $B_2$) centered at $0\in \Bbb R^p$ (resp. $\Bbb R^{n-p})$ for all $t\in [0,1]$. Let us set $\rho^2:= x_{p+1}^2+\cdots+x_n^2$. Since $X'_t$ is normally contracting to $F$, uniformly in $t\in [0,1]$, there exist positive constants $c,C$ such that
$$
-c\rho^2\leq {\cal L}_{X_t}(\rho^2)\leq -C\rho^2
$$
where $\cL_{X_t}(\rho^2)$ denotes the Lie derivative of $\rho^2$ along $X_t$.

Let $\psi$ be a smooth $\beta$-Gevrey cutting function on $\Bbb R^p$ such that $\psi\equiv 1$ in $\frac{1}{2}B_1$, $\psi\equiv 0$ on $\Bbb R^p\setminus B_1$ and $0\leq \psi\leq 1$.
Then, the vector field 
\beq\label{T} 
T_t=X'_t+\psi X''_t
\eeq
still satisfies the same kind of estimates as $X_t$~: $(*)-c\rho^2\leq \cL_{T_t}(\rho^2)\leq -C\rho^2$. Let $\phi_u^{T_t}(x)$ be the flow of $T_t$ at time $u$ passing at $x$ at $u=0$. Then, for any $x\in V':=\frac{1}{2}B_1\times B_2\times [0,1]$, the half-nonnegative trajectory $\{\phi_u^{T_t}(x), u\geq 0\}$ is contained in $V$ and, for all $t\in [0,1]$, $\lim_{u\rightarrow +\infty}\phi_u^{T_t}(x)\in B_1\times\{0\}$.

First of all, let us solve equation $(\ref{Lie-der})$. As shown by R. Roussarie, the following integral
\beq\label{sol-Lie-der}
\forall t\in[0,1],\quad f_t(x)=-\int_0^{+\infty}h_t(\phi_u^{T_t}(x))du
\eeq
defines a smooth function flat on $F$ as soon as $h_t$ is. Moreover, this function is solution of equation $\cL_{T_t}f_t=h_t$ for all $t\in [0,1]$ since we have $\cL_{T_t}f_t(\phi_u^{T_t}(x))=\frac{d f_t(\phi_u^{T_t}(x))}{du}$.

We want to show that this integral preserves the Gevrey character. According to $(*)$, we have
$$
-c\rho^2(\phi_u^{T_t}(x))\leq \frac{d \rho^2(\phi_u^{T_t}(x))}{d u}\leq -C\rho^2(\phi_u^{T_t}(x)).
$$
Hence, the derivative with respect to $u$ of $\rho(\phi_u^{T_t}(x))e^{\tfrac{cu}{2}}$ is nonnegative. Since its value at $u=0$ is $\rho(x)$, we obtain the following inequality (we proceed in the same way for the upper bound)~:
\beq\label{rho}
\rho(x)e^{-\tfrac{cu}{2}}\leq \rho(\phi_u^{T_t}(x))\leq \rho(x)e^{-\tfrac{Cu}{2}}.
\eeq

Let us show that the solution $(\ref{sol-Lie-der})$ is a smooth $\be$-Gevrey function, uniformly in $t$.
According to Komatsu theorem \cite{komatsu-edo}, for each $t$, the flow $\phi_u^{T_t}(x)$ is uniformly (in $u$) $\be$-Gevrey smooth. Therefore, by the composition lemma of Gevrey functions \cite{wagschal-goursat}, for all $t\in[0,1]$ and $u\geq 0$, $h_t(\phi_u^{T_t}(x))$ is also a smooth $\be$-Gevrey function.

We shall omit to write the depdce on $t$ and also we shall write $\phi_u(x)=\phi_u^{T_t}(x)$. As shown by R. Roussarie, the derivative of the integrant converges uniformly in $x$ in a neighborhood of the origin when $u\rightarrow +\infty$. Hence, for all $v\in \Bbb R^n$, we have
$$
D^kf(x).v^k=\int_0^{+\infty}D^k(h(\phi_u))(x).v^kdu.
$$
Let us apply Faa di Bruno formula \cite{wagschal-diff}[p.44](see also \cite{chaperon-master}), we obtain
$$
D^k(h(\phi_u))(x).v^k = \sum_{j=1}^k\sum_{\dindice{i\in (\Bbb N^*)^j}{|i|=k}}\frac{k!}{i!j!}(D^jh)(\phi_u(x)).(D^{i_1}\phi_u(x).v^{i_1},\ldots,D^{i_j}\phi_u(x).v^{i_j})
$$
where $i!=i_1!\cdots i_j!$ and $|i|=i_1+\cdots +i_j$.
Therefore, we have the following estimate
\begin{eqnarray*}
|D^k(h(\phi_u))(x).v^k| & \leq & \sum_{j=1}^k\sum_{\dindice{i\in (\Bbb N^*)^j}{|i|=k}}\frac{k!}{i!j!}\|(D^jh)(\phi_u(x))\||D^{i_1}\phi_u(x).v^{i_1}|\cdots |D^{i_j}\phi_u(x).v^{i_j}|\\
& \leq & \sum_{j=1}^k\sum_{\dindice{i\in (\Bbb N^*)^j}{|i|=k}}\frac{k!}{i!j!}\|(D^jh)(\phi_u(x))\|\|D^{i_1}\phi_u(x)\|\cdots \|D^{i_j}\phi_u(x)\||v|^{i_1+\cdots+i_j}
\end{eqnarray*}
Since $\phi_u$ is a $\beta$-Gevrey smooth function, for any compact set $K'\subset V'$, there exists a positive constant $C_{K'}$ such that $\sup_{x\in K'}\|D^{r}\phi_u(x)\|\leq C_{K'}^r(r!)^{\beta}$. Since $h$ is a smooth $\be$-Gevrey flat function along $F$, we have $(D^jh)(\pi_F(\phi_u(x)))=0$ where $\pi_F$ denotes the projection onto $F=\{z_{p+1}=\cdots=z_{n}=0\}$. 
According to lemma \ref{gevrey-flat}, we have for all $0<\la<L$, all $J\in \Bbb N^n$ such that $|J|=j$,
$$
\left|\tfrac{\partial^{j} h}{\partial x^J}(\phi_u(x))\right|\leq C\|\tfrac{\partial^{j} h}{\pa x^J}\|_{\be, L-\la;\bar V}\exp{\left( -\eta\rho(\phi_u(x))^{-\frac{1}{\be-1}}\right)}
$$
for any $0<\eta<(L-\la)^{\be/(\be-1)}$ and $C:=(1-\frac{\eta}{(L-\la)^{\be/\be-1}})^{-(\be-1)}$. Therefore, according to lemma \ref{deriv-sauzin}, we obtain
\begin{eqnarray*}
\sum_{J\in \Bbb N^n, |J|=j}\left|\tfrac{\partial^{|J|} h}{\partial x^J}(\phi_u(x))\right|&\leq &
C\exp{\left( -\eta\rho(\phi_u(x))^{-\frac{1}{\be-1}}\right)}\sum_{J\in \Bbb N^n, |J|=j}\left\|\tfrac{\pa^{|J|}f}{\pa^J y}\right\|_{\be,L-\la;\bar V}\\
&\leq & C\exp{\left( -\eta\rho(\phi_u(x))^{-\frac{1}{\be-1}}\right)}j!^{\be}\la^{-j\be}\|h\|_{\be,L;\bar V}
\end{eqnarray*}
Moreover, according to inequality $(\ref{rho})$, we have, for all $k\in\Bbb N$,
$$
\exp\left(-\eta\rho(\phi_u(x))^{-\tfrac{1}{\be-1}}\right)\leq \exp\left(-\eta\left(\rho(x)\exp(-\tfrac{Cu}{2})\right)^{-\tfrac{1}{\be-1}}\right)\leq \left(\frac{\rho(x)\exp(-\tfrac{Cu}{2})}{\eta^{\be-1}}\right)^{\frac{k}{\be-1}}
$$
Using the previous estimate with $k=j$, we finally obtain the following estimate
$$
|D^k(h(\phi_u))(x).v^k|\leq |v|^{k}\|h\|_{\be,L;\bar V}C_{K'}^kC\exp(-\tfrac{Cu}{2})\sum_{j=1}^k\sum_{\dindice{i\in (\Bbb N^*)^j}{|i|=k}}\frac{k!}{i!j!}\left(\frac{\rho(x)^{\tfrac{1}{\be-1}}}{\eta\la^{\be}}\right)^j(j!)^{\be}(i!)^{\be}
$$
According to lemma 2.6 of \cite{shin}, there exists a constant $H_{\be}$ such that
$$
\sum_{\dindice{i\in (\Bbb N^*)^j}{|i|=k}}\frac{(i!)^{\be-1}}{(k!) 
^{\be-1}}\leq \frac{H_{\be}^j}{(j!)^{\be-1}}.
$$
Therefore, we obtain
$$
\sum_{j=1}^k\sum_{\dindice{i\in (\Bbb N^*)^j}{|i|=k}}\frac{k!}{i!j!}D_K^j(j+1!)^{\be}(i!)^{\be}  \leq 
(k!)^{\be}\sum_{j=1}^k \left(\frac{H_{\be}\rho(x)^{\tfrac{1}{\be-1}}}{\eta\la^{\be}}\right)^j
$$
Since the series $\sum_{j=1}^k \left(\frac{H_{\be}\rho(x)^{\tfrac{1}{\be-1}}}{\eta\la^{\be}}\right)^j$ is converging if $\rho(x)$ is small enough. Hence, we finally obtain
\begin{equation}
|D^k(h(\phi_u))(x).v^k| \leq  (k!)^{\be}|v|^{k}E_{K'}^ke^{-\tfrac{Cu}{2}}\label{estim1}
\end{equation}
for some positive constant $E_{K'}$. Therefore, $|D^k(h(\phi_u))(x).v^k|$ is uniformly convergent (w.r.t. $x$ in $K'$ in a small enough neighborhood) when $u\rightarrow +\infty$. Hence, the resulting function $f$ is a smooth $\be$-Gevrey flat function along $F$ on a neighborhood $V''$ of the origin.\vspace{.3 cm}\\

Let us consider the cohomological equation $(\ref{Lie-bra})$. In order to solve it in neighborhood of the origin, it is sufficient to find a $1$-parameter family $\{Z_t\}_{t\in[0,1]}$ of germs smooth $\be$-Gevrey vector fields such that $(*)\;[T_t,Z_t]=Y_t$ where $Y_t$ a given $1$-parameter family of germs smooth $\be$-Gevrey vector fields at the origin which are flat along $F$ and where $T_t$ is defined by $(\ref{T})$. As above, we omit the dependence in $t\in [0,1]$. Let $\ga(x,u)$ denotes the flow of $T$ at time $u$ passing through $x$ at $u=0$. Following \cite{dumortier-book}, let us evaluate equation $(*)$ at the point $\ga(x,u)$. We obtain~:
$$
DZ(\ga(x,u))T(\ga(x,u))- DT(\ga(x,u))Z(\ga(x,u))=Y(\ga(x,u)).
$$
Let us set $V(x,u):=Z(\ga(x,u))$. Then it satisfies to
$$
\frac{dV(x,u)}{du} = DT(\ga(x,u))V(x,u)+Y(\ga(x,u)).
$$
The classical theory of linear differential equations (cf. \cite{wasow} for instance) then gives
$$
V(x,u) = -\int_0^{\infty}F(x,\tau)^{-1}Y(\ga(x,\tau))d\tau
$$
where $F(x,\tau)$ is the fundamental solutions matrix of the linear homogeneous system~:
$$
\frac{d F(x,\tau)}{d\tau}=DT(\ga(x,\tau))F(x,\tau),\quad F(x,0)=\Id.
$$
According to Komatsu theorem \ref{komatsu-edo}, $\ga(x,u)$ is a $\be$-Gevrey function in $x$, uniformly in $u$. Hence, $DT(\ga(x,\tau))$ is also $\be$-Gevrey smooth as the composition of two $\be$-Gevrey smooth maps. Thus, again according to Komatsu theorem, both $F(x,\tau)$ and $F(x,\tau)^{-1}$ are smooth $\be$-Gevrey matrix-valued functions (uniformly in $\tau$). As a consequence, the components of $F(x,\tau)^{-1}Y(\ga(x,\tau))$ are a finite sum of function of the form $f(x,\tau)y(\ga(x,\tau))$ where $f(x,\tau)$ is a smooth $\be$-Gevrey function and $y$ is a smooth $\be$-Gevrey function flat on $F$. Combining estimate $(\ref{estim1})$ together with Leibniz formula, we find that
\begin{eqnarray*}
|D^k(f(x,\tau)y(\ga(x,\tau)).v^k|&\leq & \sum_{k_1+k_2=k}\frac{k!}{k_1!k_2!}|D^{k_1}(f(x,\tau)).v^{k_1}||D^{k_2}(y(\ga(x,\tau)).v^{k_2}|\\
&\leq & \sum_{k_1+k_2=k}\frac{k!}{k_1!k_2!}(k_1!)^{\be}|v|^{k_1}F_K^{k_1}(k_2!)^{\be}|v|^{k_2}E_K^{k_2}\rho(x)e^{-\tfrac{Cu}{2}}\\
& \leq & G_K^{k}|v|^{k}(k!)^{\be}\rho(x)e^{-\tfrac{Cu}{2}}
\end{eqnarray*}
for some positive constant $G_K$. Therefore, as above, $\int_0^{\infty}f(x,\tau)y(\ga(x,\tau)d\tau$ defines a smooth $\be$-Gevrey function, flat on $F$, in a neighborhood of the origin. As a consequence, the solution of equation $(*)$ is also a smooth $\be$-Gevrey vector field, flat on $F$, in a neighborhood of the origin.
This ends the proof of proposition \ref{prop-roussarie}.
\end{proof}

\subsection{Gevrey stable and unstable manifolds}

It is well known that a smooth vector field near an hyperbolic fixed point admits a germ of smooth stable and unstable manifold which are tangent to generalized eigenspaces $E^s$ (resp. $E^u$) associated to eigenvalues with negative real part (resp. positive real parts). In this section, we shall show that these manifolds are $\be$-Gevrey smooth as soon as the vector field considered is also $\be$-Gevrey smooth ($\be\geq 1$). 

\begin{theo}\label{stable-manifold}
Let us consider the germ of smooth $\be$-Gevrey vector field at the origin, an hyperbolic fixed point~: 
\beq\label{nonlin}
\frac{dx}{dt}= Ax+g(x)
\eeq
where $A$ is an hyperbolic matrix, $g$ a smooth $\be$-Gevrey vector field in a neighborhood of $0$ with $g(0)=0$, $Dg(0)=0$.
Let $\phi^t(w)$ denote its flow. Then, there exists a neighborhood ${\cal V}$ of the origin such that
$$
W^s_{\text{loc}}=\{x\in {\cal V}\,|\,\{t\mapsto \phi^t(x)\}\in C^1_b(\Bbb R^+)\},\quad W^u_{\text{loc}}=\{x\in {\cal V}\,|\,\{t\mapsto \phi^t(x)\}\in C^1_b(\Bbb R^-)\}
$$
are smooth $\be$-Gevrey manifolds passing through $0$ and tangent respectively to $E^s$, $E^u$. These are the local stable and unstable manifolds respectively.
\end{theo}
\begin{proof}
We follow the proof of Benzoni-Cavage \cite{benzoni-cours}[p.42-45]. We consider the case of the local unstable manifold. Let us define $Z:=C_b^0(\Bbb R^-)$ (resp. $X:=C_b^1(\Bbb R^-)$) to be the space of bounded continuous functions (resp. with bounded derivative) on $\Bbb R^-$,  and the linear continuous operator $B: x\in X\mapsto \frac{dx}{dt}-Ax\in Z$. Let us define
\begin{eqnarray*}
 S : E^u & \rightarrow & X\\
w & \mapsto & Sw =\{t\mapsto e^{tA}w\}
\end{eqnarray*}
and let $P : X\rightarrow X$ be the projector onto $Ker B$. A solution $x$ to $(\ref{nonlin})$ satisfying to $x(0)=w$ belongs to $X$ if and only if $Px=S\pi_u(w)$ and $x_1=x-Px\in X_1:=Ker P$ solves $Bx_1=g(S\pi_u(w)+x_1)$. Then, we can show \cite{benzoni-cours} that there exists a neighborhood ${\cal V}_1\times {\cal W}$ of $(0,0)$ in $X_1\times E^u$ and a $C^1$ map $\chi$ from ${\cal W}$ to ${\cal V}_1$ such that for all $w\in \pi_u^{-1}({\cal W})$, equation $(\ref{nonlin})$ admits a unique solution $x\in X$ with $x(0)=w$ given by 
$$
x=S\zeta+\chi_1(\zeta),\;\text{with}\; \zeta=\pi_u(w)
$$ 
Moreover, we have $\chi_1(0)=0$ and $D_{\zeta}\chi_1(0)=0$. According to theorem \ref{komatsu-edo} \cite{komatsu-edo}, since $g$ is $\be$-Gevrey in a neighborhood of the origin, then $x$ is also $\be$-Gevrey in $w$ uniformly in $t$. In particular, the map $\zeta\mapsto\chi_1(\zeta)(0)$ is $\be$-Gevrey in $\zeta$. Therefore, the local unstable manifold 
$$
\{w\in\pi^{-1}_u({\cal W})\;|\;t\mapsto\phi^t(w)\in X\}=\{w=S\zeta+\chi_1(\zeta)(0); \zeta\in {\cal W}\}
$$
is a smooth $\be$-Gevrey manifold. The local stable manifolds is obtained considering $X:=C_b^1(\Bbb R^+)$, $Z:=C_b^0(\Bbb R^+)$ instead.
\end{proof}
\begin{rem}\label{rem-var}
We can apply a linear change of coordinates so that $E^u=\{x_1=\cdots=x_p=0\}$ and $E^s=\{x_{p+1}=\cdots=x_n=0\}$. As a consequence, there are local $\be$-Gevrey coordinates $(y_1,\ldots, y_n)$ such that $\cW_{loc}^s=\{y_{p+1}=\cdots=y_n=0\}\cap U$ and $\cW_{loc}^u=\{y_{1}=\cdots=y_p=0\}\cap U$ for some neighborhood $U$ of the origin.
\end{rem}
\section{Proof of theorem \ref{theo-flat}}

Let $\{X_t\}_{t\in [0,1]}$ be a $1$-parameter family of germs of smooth $\be$-Gevrey vector fields at the origin of $\Bbb R^n$ vanishing of $F$ (defined on same neighborhood of the origin). Assume that $\{X_t\}_{t\in [0,1]}$ is non-degenerate hyperbolic transversally to $F$, uniformly in $t\in[0,1]$. Let us apply the Gevrey local stable and unstable manifold theorem \ref{stable-manifold} as well as remark \ref{rem-var}~: there is a family of good $\be$-Gevrey change of variables in which $F_1=\{y_1=\cdots=y_p=0\}$ and $F_2=\{y_{p+1}=\cdots=y_n=0\}$ are the unstable and stable manifolds of $X_t$ respectively. 

Let us decompose $X_t$ as $X_t=X_t'+X_t''$ where $X'_t=\sum_{i=p+1}^nX_{t,i}(y)\frac{\partial}{\partial y_i}$ (resp. $X_t''=\sum_{i=1}^pX_{t,i}(x)\frac{\partial}{\partial y_i}$) is parallel to $F_1$ (resp. $F_2$).

Let us first solve equation $(\ref{Lie-der})$~: Let $f_t$ be a family of germs of smooth $\be$-Gevrey functions flat at the origin. Next lemma allow us to decompose the family $f_t=f_{t,1}+f_{t,2}$ as the sum of two families $f_{t,1}$ and $f_{t,2}$ which are flat on $F_1$ and $F_2$ respectively. It is a Gevrey version of lemma 4 of \cite{ilya-yakov-russian}~:
\begin{lemm}\label{decompo}
 Let $\Bbb R^n=E^{s}\oplus E^{u}$ be an arbitrary decomposition of $\Bbb R^n$ into direct sums. Let $f\in {\cal G}^{<-\beta}$ be a germ of smooth $\beta$-Gevrey function flat at the origin. Then, there exist  smooth $\beta$-Gevrey functions $f_s,f_u$ in a neighborhood of the origin which are flat on $E^{s}$ and $E^{u}$ respectively such that $f=f_s+f_u$.
\end{lemm}
\begin{proof}
Let $K$ be the compact set defined to be the intersection of a neighborhood of the origin with $E^{s}\cup E^{u}$. Let $\ti f=(f^k)_{k\in\Bbb N^n}$ be the jet defined to be $f^k(x)=0$ if $x\in K\cap E^{s}$ and $f^k(x)=D^kf(x)$ if $x\in K\cap E^{u}$. Since $f$ is flat at $0$, the functions $f^k$ are continuous. Moreover, $\ti f$ is a $\be$-Gevrey-Whitney jet on $K$ since $f$ is a smooth $\be$-Gevrey function. Thus, according to theorem \ref{bruna}, there exists a smooth $\be$-Gevrey function $f^s$ such that $D^kf^s(x)=f^k(x)$ on $K$. Hence, it is flat on $E^{s}\cap K$. Let us define $f^u=f-f^s$. It is a smooth $\be$-Gevrey function which is flat on $E^{u}\cap K$.
\end{proof}

Now, $-X_t'$ is normally contracting to $F_2$ while $X''_t$ is normally contracting to $F_1$. In fact, we have 
$$
X_t=\sum_{i=1}^p\left(\sum_{j=1}^pX_{t,i,j}(y)y_j\right)\frac{\partial }{\partial y_i}+\sum_{i=p+1}^n\left(\sum_{j=p+1}^nX_{t,i,j}(y)y_j\right)\frac{\partial }{\partial y_i}
$$
and the matrices $(X_{t,i,j}(0))_{1\leq i,j\leq p}$ and $(X_{t,i,j}(0))_{p+1\leq i,j\leq n}$ have eigenvalues with negative and positive real parts respectively.
Therefore, according to proposition \ref{prop-roussarie}, equation $\cL_{-X'_t-X''_t}(h'_t)=-f_{t,2}$ has a $\be$-Gevrey solution $h'_t$ which is flat on $F_2$. On the other hand, equation $\cL_{X''_t+X'_t}(h''_t)=f_{t,1}$ has a $\be$-Gevrey solution $h''_t$ which is flat on $F_1$. As a consequence, we have $\cL_{X_t}(h'_{t}+h''_t)=f_t$ and $h_t:=h'_{t}+h''_t$ is a family of germs of $\be$-Gevrey functions flat at the origin.

We do the same reasoning in order to solve $(\ref{Lie-bra})$  : we decompose the family of $\be$-Gevrey flat vector fields $Y_t$ as the sum of $Y_{t,1}$ and $Y_{t,2}$ which are $\be$-Gevrey flat on $F_1$ and $F_2$ respectively. We the apply the second part of proposition \ref{prop-roussarie}. We are done.

\appendix
\section{Gevrey functions}\label{gevrey}

In this section, we recall some result that we will use about Gevrey functions. First of all, let $\Om$ be an open set $\Bbb R^n$ and $\al\geq 1$. A smooth complex-valued function $f$ on an open set $\Om$ of $\Bbb R^n$ is said to be {\bf $\al$-Gevrey} if for any compact set $K\subset \Om$, there exist constants $M$ and $C$ such that, for all $k\in\Bbb N^n$,
$$
\sup_{x\in K}|D^{k}f(x)|\leq MC^{|k|}|k|!^{\al}.
$$
We refer to \cite{rodino-book} for more information about Gevrey functions. Following Marco-Sauzin \cite{sauzin-ihes},
\begin{defi}\label{gevrey-def}
Let $K$ be a compact set in $\Om$. We define the Banach algebra of complex valued {\bf $\al$-Gevrey functions on $K$ of width $L$} to be
\begin{eqnarray*}
{\cal G}_{\al, L;K}&:=&\left\{f\in C^{\infty}(\Om;\Bbb C)\;|\;\|f\|_{\al ,L;K}<+\infty\right\}\;\;\text{with}\\
\|f\|_{\al,L;K}&:=&\sum_{Q\in\Bbb N^n}\frac{L^{\al |Q|}}{q!^{\al}}\left\|\tfrac{\pa^{|Q|}f}{\pa^Q y}\right\|_{C^0(K)}.
\end{eqnarray*}
Let $f$ be a smooth complex-valued $\al$-Gevrey function in an open set $\Om$ of $\Bbb R^n$. Then, for any compact subset $K$ of $\Om$, there exists $L_K>0$ such that $\|f\|_{\al,L_K;K}<+\infty$.
\end{defi}
The following lemma will be useful~:
\begin{lemm}[Lemma A.2]\cite{sauzin-ihes}\label{deriv-sauzin}
 Let $\al\geq 1$ and $\phi\in \cG_{\al,L;K}$. Then, for any $0<\la<L$, all partial derivative of $\phi$ belong to $\cG_{\al,L-\la;K}$. Moreover, we have, for all $k\in \Bbb N$,
\beq\label{deriv-estim}
\sum_{Q\in \Bbb N^n, |Q|=j}\left\|\tfrac{\pa^{|Q|}\phi}{\pa^Q y}\right\|_{\al,L-\la;K}\leq j!^{\al}\la^{-j\al}\|\phi\|_{\al,L;K}
\eeq
\end{lemm}
Let us recall a classical result about flatness of Gevrey functions. For safe of completeness, we give a proof of D. Sauzin~:
\begin{prop}\label{gevrey-flat}
Let $\al>1$. Let $\cK\subset \cL$ be compact subsets of $\Om$, with $\cL$ convex. Let $f\in {\cal G}_{\al, L;\cL}$ be a smooth $\al$-Gevrey function flat on $\cK$, that is $\tfrac{\pa^{|Q|}f}{\pa^Q y}(y_0)= 0$ for all $Q\in \Bbb N^n$, for all $y_0\in\cK$. Then, for all $y\in \cL$ such that $\dist(y,\cK)=\ep$, we have
$$
|f(y)|\leq C\|f\|_{\al, L;\cL}\exp\left(-\la \ep^{-\frac{1}{\al-1}}\right)
$$
for any $0<\la<L^{\al/\al-1}$ and where $C=(1-\frac{\la}{L^{\al/\al-1}})^{-(\al-1)}$.
\end{prop}
\begin{proof}
Let $y_0$ be a point of $K$ and let $y\in \cL$ be such that $\|y_0-y\|=\ep$. For each $j\in\Bbb N$, let us Taylor expand $f(y)$ at the point $y_0$ at order $j$. We have
$$
f(y)=\sum_{Q\in \Bbb N^m,\; |Q|=k}\frac{(y-y_0)^Q}{Q!}\int_0^1\tfrac{\pa^{k}f}{\pa^Q y}(y_0+t(y-y_0))k(1-t)^k dt
$$
since all the derivatives of $f$ at $y_0$ vanish.
Therefore, we have
\begin{eqnarray*}
|f(y)| &\leq & \ep^k\sum_{Q\in \Bbb N^m,\; |Q|=k}\frac{\left\|\tfrac{\pa^{k}f}{\pa^Q y}\right\|_{C^0(\cL)}}{Q!}\\
&\leq & \ep^k\sum_{Q\in \Bbb N^m,\; |Q|=k}\frac{L^{k\al}}{Q!^{\al}}\|\tfrac{\pa^{k}f}{\pa^Q y}\|_{C^0(\cL)}Q!^{\al-1}L^{-k\al}.
\end{eqnarray*}
Since $Q!\leq k!$ in the previous sum, we obtain the following estimate
$$
|f(y)|\leq  \ep^kk!^{\al-1}L^{-k\al}\|f\|_{\al,L;\cL}.
$$
Therefore, we obtain
$$
\frac{\ep^{-k/(\al-1)}\la^k}{k!}|f(y)|^{1/(\al-1)}\leq  \left(\frac{\la}{L^{\al/(\al-1)}}\right)^{k}\|f\|_{\al,L;\cL}^{1/(\al-1)}.
$$
Then summing over the $k$'s yields
$$
\exp\left(\la \ep^{-\frac{1}{(\al-1)}}\right)|f(y)|^{1/(\al-1)}\leq  C^{1/\al-1}\|f\|_{\al,L;\cL}^{1/(\al-1)}.
$$
\end{proof}

\begin{defi}
Let $\al\geq 0$. A formal power series $\hat f=\sum_{Q\in \Bbb N^n}f_Qx^Q\in \Bbb C[[x_1,\ldots,x_n]]$ is said to be $\al$-Gevrey if there exist positive constants $M,C$ such that, for all $Q\in \Bbb N^n$, 
$|f_Q|\leq M C^{|Q|}(|Q|!)^{\al}$.
\end{defi}
First, let us recall the definition of Gevrey-Whitney jets as well as the Gevrey-Whitney extension theorem. These are due to J. Bruna (see also \cite{popov-herman})~:
\begin{defi}\cite{bruna}
Let $\al\geq 1$. Let $K$ be a compact set in $\Bbb R^n$. A $\al$-Gevrey-Whitney jet is a collection $F=(f^k)_{k\in\Bbb N^n}$ of continuous functions such that there exist $C>0$ and $M>0$ and 
\begin{enumerate}
 \item $|f^k(x)|\leq MC^{|k|}(|k|!)^{\al}$, for all $k\in\Bbb N^n$ and $x\in K$;
. \item $|(R_x^mF)_k(y)|\leq M \frac{|x-y|^{m-|k|+1}}{(m-|k|+1)!}C^{m+1}(m+1)!^{\al}$, for all $x,y\in K$, $m\in \Bbb N$, $|k|\leq m$
\end{enumerate}
where 
$$
(R_x^mF)_k(y):=f^k(y)-\sum_{|k+j|\leq m}\frac{f^{k+j}(x)}{j!}(y-x)^j.
$$
\end{defi}
\begin{rem}
 If $K$ is reduced to a point, say $0$, then a $\al$-Gevrey-Whitney jet is just the $(\al-1)$-Gevrey formal power series $\sum_{Q\in \Bbb N^n}\frac{f^Q(0)}{Q!}x^Q$.
\end{rem}

\begin{theo}\cite{bruna}\label{bruna}
 Let $F$ be a $\al$-Gevrey-Whitney jet on a compact set $K$. Then, there exists a $\al$-Gevrey smooth function $f$ on $\Bbb R^n$ such that $D^kf(x)=f^k(x)$ for all $x\in K$ and $k\in \Bbb N^n$.
\end{theo}
Let us recall the theorem of Komatsu about solutions of Gevrey differential equations~:
\begin{theo}\cite{komatsu-edo}\label{komatsu-edo}
Let $f_1(t,x),\ldots,f_n(t,x)$ be smooth functions on $]-T,T[\times \Om$, where $\Om$ is an open subset of $\Bbb R^n$. We assume that they are $\al$-Gevrey functions in $x$ uniformly in $t$. 
Then, the initial value problem
$$
\frac{dx_i}{dt}=f_i(t,x),\quad i=1,\ldots, n
$$
and $x(0)=y\in \Om$ admits a unique solution $x(t,y)$ which is $\al$-Gevrey in $y$ on $\Om$ uniformly in $t$.
\end{theo}

\bibliographystyle{alpha}
\bibliography{normal,math,asympt,analyse,stolo,cr,normal-bis}

\end{document}

%% file: macro_L1.tex
\def\al{\alpha}
\def\be{\beta}
\def\ga{\gamma}

\def\de{\delta}

\def\ep{{\epsilon}}

\def\la{\lambda}

\def\om{\omega}

\def\Om{\Omega}

\def\sig{{\sigma}}

\def\ze{{\zeta}}

\newcommand{\dist}{\operatorname{dist}}

\def\pa{\partial}

\newcommand{\ti}{\tilde}

\newcommand{\R}{\mathbb{R}}

\def\cG{\mathcal{G}}
\newcommand{\cH}{\mathcal{H}}

\def\cK{\mathcal{K}}

\def\cL{\mathcal{L}}

\def\cW{\mathcal{W}}

\def\cc#1{_{(#1)}}

\def\om{\omega}



%% file: macro_L2.tex
\def\ftoday{le \space\number\day \space\ifcase\month\or
  janvier\or f\'evrier\or mars\or avril\or mai\or juin\or
  juillet\or ao\^ut\or septembre\or octobre\or novembre\or d\'ecembre\fi
  \space\number\year}

\def\real{I\kern-0.20em R}

\def\integer{I\kern-0.20em N}
\def\relative{{\rm \rlap Z\kern 2.2pt Z}}
\def\cc{\kern-.25em{\c c}}

\def\bc{\begin{center}}
\def\ec{\end{center}}
\def\=def{\stackrel{{\rm def}}{=}}

\newcounter{indconst}
\newcounter{auxconst}

\def\bit{\begin{itemize}}
\def\eit{\end{itemize}}
\def\ben{\begin{enumerate}}
\def\een{\end{enumerate}}
\def\bde{\begin{description}}
\def\ede{\end{description}}
 
\def\beq{\begin{equation}}
\def\eeq{\end{equation}}
\def\bfi{\begin{figure}[hbt] \begin{center}}
\def\efi{\end{center} \end{figure}}

\def\bce{\begin{center}}
\def\ece{\end{center}}

\newtheorem {theo} {Theorem}[section]
\newtheorem {coro} [theo]{Corollary} 
\newtheorem {lemm}[theo] {Lemma}
\newtheorem {prop}[theo]{Proposition}
\newtheorem {defi}[theo] {Definition}
\newtheorem {rem}[theo]{Remark}

%% file: macro_E.tex
\newcommand{\suml}{\sum\limits}

\newcommand{\dindice}[2]{{\stackrel{\scriptstyle #1}{\scriptstyle #2}}}

\renewcommand{\epsilon}{\varepsilon}

\newcommand{\Id}{{\rm Id}}

\newcommand{\ssi}[1]{si et seulement si}

\newcommand{\Norme}[2]%
  {\left |  #1 \right |_{\hspace{-0.9ex}\raisebox{-0.5ex}{ $\scriptstyle
{#2}$}}}
\newcommand{\NNorme}[2]%
{\left \|  #1 \right \|_{\hspace{-0.9ex}\raisebox{-0.5ex}{ $\scriptstyle
#2  $}}}

\newcommand{\norme}[2]%
  {\left |  #1 \right |_{\hspace{-0.9ex}\raisebox{-0.5ex}{ $\scriptscriptstyle
{#2}$}}}
\newcommand{\nnorme}[2]%
{\left \|  #1 \right \|_{\hspace{-0.9ex}\raisebox{-0.5ex}{ $\scriptscriptstyle
#2  $}}}

\newcommand{\sNorme}[2]%
  { |  #1 |_{\hspace{-0.9ex}\raisebox{-0.5ex}{ $\scriptstyle #2$}}}
\newcommand{\sNNorme}[2]%
{ \|  #1  \|_{\hspace{-0.9ex}\raisebox{-0.5ex}{ $\scriptstyle #2$}}}

\newcommand{\bNorme}[2]%
  { \Bigl |  #1 \Bigr |_{\hspace{-0.9ex}\raisebox{-0.5ex}{ $\scriptstyle #2$}}}
\newcommand{\bNNorme}[2]%
{ \Bigl \|  #1 \Bigr   \|_{\hspace{-0.9ex}\raisebox{-0.5ex}{ $\scriptstyle #2$}}}

\newcommand{\scalxx}[2]{\left \langle #1, #2 \right   \rangle}
\newcommand{\scalx}[3]%
{\scalxx{#1}{#2}_{\hspace{-0.4ex}\raisebox{-0.3ex}{$\scriptstyle {#3}$}}}

%% file: stolovitch-gevrey5.bbl
\newcommand{\etalchar}[1]{$^{#1}$}
\def\cprime{$'$} \def\cprime{$'$}
\begin{thebibliography}{ETB{\etalchar{+}}87}

\bibitem[Arn80]{Arn2}
V.I. Arnold.
\newblock {\em {Chapitres suppl\'{e}mentaires de la th\'{e}orie des
  \'{e}quations diff\'{e}rentielles ordinaires}}.
\newblock Mir, 1980.

\bibitem[Bel79]{belitskii-nf}
G.~R. Belicki{\u\i}.
\newblock Invariant normal forms of formal series.
\newblock {\em Funktsional. Anal. i Prilozhen.}, 13(1):59--60, 1979.

\bibitem[Ben07]{benzoni-cours}
S.~Benzoni.
\newblock {\em Equations diff{\'e}rentielles ordinaires}, 2007.
\newblock http://math.univ-lyon1.fr/~benzoni/edo.html.

\bibitem[Bru72]{bruno}
A.D. Bruno.
\newblock {Analytical form of differential equations}.
\newblock {\em Trans. Mosc. Math. Soc}, 25,131-288(1971); 26,199-239(1972),
  1971-1972.

\bibitem[Bru80]{bruna}
Joaquim Bruna.
\newblock An extension theorem of {W}hitney type for non-quasi-analytic classes
  of functions.
\newblock {\em J. London Math. Soc. (2)}, 22(3):495--505, 1980.

\bibitem[BS07]{stolo-boele}
B.~Braaksma and L.~Stolovitch.
\newblock Small divisors and large multipliers.
\newblock {\em Ann. Inst. Fourier (Grenoble)}, 57(2):603--628, 2007.

\bibitem[Car03]{carletti}
Timoteo Carletti.
\newblock The {L}agrange inversion formula on non-{A}rchimedean fields.
  {N}on-analytical form of differential and finite difference equations.
\newblock {\em Discrete Contin. Dyn. Syst.}, 9(4):835--858, 2003.

\bibitem[Cha86a]{chaperon-ihes}
M.~Chaperon.
\newblock $c^k$-conjugacy of holomorphic flows near a singularity.
\newblock {\em Publ. Math. I.H.E.S.}, (64):143--183, 1986.

\bibitem[Cha86b]{chaperon-ast}
M.~Chaperon.
\newblock G\'eom\'etrie diff\'erentielle et singularit\'es de syst\`emes
  dynamiques.
\newblock {\em Ast\'erisque}, 138-139, 1986.

\bibitem[Cha08]{chaperon-master}
Marc Chaperon.
\newblock {\em Calcul diff{\'e}rentiel et calcul int{\'e}gral}.
\newblock Dunod, 2008.

\bibitem[CM00]{carletti-marmi}
Timoteo Carletti and Stefano Marmi.
\newblock Linearization of analytic and non-analytic germs of diffeomorphisms
  of {$({\bf C},0)$}.
\newblock {\em Bull. Soc. Math. France}, 128(1):69--85, 2000.

\bibitem[DLA06]{dumortier-book}
Freddy Dumortier, Jaume Llibre, and Joan~C. Art{\'e}s.
\newblock {\em Qualitative theory of planar differential systems}.
\newblock Universitext. Springer-Verlag, Berlin, 2006.

\bibitem[ETB{\etalchar{+}}87]{iooss-elphick}
C.~Elphick, E.~Tirapegui, M.~E. Brachet, P.~Coullet, and G.~Iooss.
\newblock A simple global characterization for normal forms of singular vector
  fields.
\newblock {\em Phys. D}, 29(1-2):95--127, 1987.

\bibitem[Fis17]{fischer}
E.~Fischer.
\newblock \"{U}ber die {D}ifferentiationsprozesse der {A}lgebra.
\newblock {\em J. f\"ur Math. 148, 1-78.}, 148:1--78, 1917.

\bibitem[Fra95]{francoise-book}
J.-P. Fran{\c{c}}oise.
\newblock {\em G\'eom\'etrie analytique et syst\`emes dynamiques}.
\newblock Presses Universitaires de France, Paris, 1995.

\bibitem[IL05]{IoossLombardi}
G.~Iooss and E.~Lombardi.
\newblock Polynomial normal forms with exponentially small remainder for
  analytic vector fields.
\newblock {\em J. Differential Equations}, 212(1):1--61, 2005.

\bibitem[Ito89]{Ito1}
H.~Ito.
\newblock Convergence of birkhoff normal forms for integrable systems.
\newblock {\em Comment. Math. Helv.}, 64:412--461, 1989.

\bibitem[IY91]{ilya-yakov-russian}
Yu.~S. Il{\cprime}yashenko and S.~Yu. Yakovenko.
\newblock Finitely smooth normal forms of local families of diffeomorphisms and
  vector fields.
\newblock {\em Uspekhi Mat. Nauk}, 46(1(277)):3--39, 240, 1991.

\bibitem[Kom79]{komatsu-inverse}
Hikosaburo Komatsu.
\newblock The implicit function theorem for ultradifferentiable mappings.
\newblock {\em Proc. Japan Acad. Ser. A Math. Sci.}, 55(3):69--72, 1979.

\bibitem[Kom80]{komatsu-edo}
Hikosaburo Komatsu.
\newblock Ultradifferentiability of solutions of ordinary differential
  equations.
\newblock {\em Proc. Japan Acad. Ser. A Math. Sci.}, 56(4):137--142, 1980.

\bibitem[LS09]{stolo-lombardi-cras}
E.~Lombardi and L.~Stolovitch.
\newblock Forme normale de perturbation de champs de vecteurs
  quasi-homog\`enes.
\newblock {\em C.R. Acad. Sci, Paris, S\'{e}rie I}, 347:143--146, 2009.

\bibitem[LS10]{stolo-lombardi}
E.~Lombardi and L.~Stolovitch.
\newblock Normal forms of analytic perturbations of quasihomogeneous vector
  fields: Rigidity, invariant analytic sets and exponentially small
  approximation.
\newblock {\em Ann. Scient. Ec. Norm. Sup.}, pages 659--718, 2010.

\bibitem[MS02]{sauzin-ihes}
Jean-Pierre Marco and David Sauzin.
\newblock Stability and instability for {G}evrey quasi-convex near-integrable
  {H}amiltonian systems.
\newblock {\em Publ. Math. Inst. Hautes \'Etudes Sci.}, (96):199--275 (2003),
  2002.

\bibitem[Pop00]{popov-ihp1}
G.~Popov.
\newblock Invariant tori, effective stability, and quasimodes with
  exponentially small error terms. {I}. {B}irkhoff normal forms.
\newblock {\em Ann. Henri Poincar\'e}, 1(2):223--248, 2000.

\bibitem[Pop04]{popov-herman}
G.~Popov.
\newblock K{AM} theorem for {G}evrey {H}amiltonians.
\newblock {\em Ergodic Theory Dynam. Systems}, 24(5):1753--1786, 2004.

\bibitem[Rod93]{rodino-book}
Luigi Rodino.
\newblock {\em Linear partial differential operators in {G}evrey spaces}.
\newblock World Scientific Publishing Co. Inc., River Edge, NJ, 1993.

\bibitem[Rou75]{roussarie-ast}
R.~Roussarie.
\newblock Mod\`{e}les locaux de champs et de formes.
\newblock {\em Ast\'{e}risque}, 30, 1975.

\bibitem[SCK03]{shin}
Chang~Eon Shin, Soon-Yeong Chung, and Dohan Kim.
\newblock Gevrey and analytic convergence of {P}icard's successive
  approximations.
\newblock {\em Integral Transforms Spec. Funct.}, 14(1):19--30, 2003.

\bibitem[Sha89]{shapiro}
H.~S. Shapiro.
\newblock An algebraic theorem of {E}.\ {F}ischer, and the holomorphic
  {G}oursat problem.
\newblock {\em Bull. London Math. Soc.}, 21(6):513--537, 1989.

\bibitem[Sie42]{siegel}
C.L. Siegel.
\newblock {Iterations of analytic functions}.
\newblock {\em Ann. Math.}, 43(1942)807-812, 1942.

\bibitem[Ste58]{stern2}
S.~Sternberg.
\newblock On the structure of local homeomorphisms of euclidean {$n$}-space.
  {II}.
\newblock {\em Amer. J. Math.}, 80:623--631, 1958.

\bibitem[Sto94]{stolo-dulac}
L.~Stolovitch.
\newblock {Sur un th\'{e}or\`{e}me de Dulac}.
\newblock {\em Ann. Inst. Fourier}, 44(5):1397--1433, 1994.

\bibitem[Sto00]{Stolo-ihes}
L.~Stolovitch.
\newblock Singular complete integrabilty.
\newblock {\em Publ. Math. I.H.E.S.}, 91:p.133--210, 2000.

\bibitem[Sto05]{Stolo-cartan}
L.~Stolovitch.
\newblock Normalisation holomorphe d'alg\`ebres de type {C}artan de champs de
  vecteurs holomorphes singuliers.
\newblock {\em Ann. of Math.}, (161):589--612, 2005.

\bibitem[Sto08]{Stolo-asi07}
L.~Stolovitch.
\newblock {Normal Forms of holomorphic dynamical systems}.
\newblock In W.~Craig, editor, {\em Hamiltonian dynamical systems and
  applications}, pages 249--284. Springer-Verlag, 2008.

\bibitem[Sto09]{Stolo-nonlin}
L.~Stolovitch.
\newblock Progress in normal form theory.
\newblock {\em Nonlinearity}, 22:R77--R99, Invited article, 2009.

\bibitem[Vey78]{vey-ham}
J.~Vey.
\newblock Sur certains syst\`{e}mes dynamiques s\'{e}parables.
\newblock {\em Am. Journal of Math. 100}, pages 591--614, 1978.

\bibitem[Vey79]{vey-iso}
J.~Vey.
\newblock Alg\`{e}bres commutatives de champs de vecteurs isochores.
\newblock {\em Bull. Soc. Math. France,107}, pages 423--432, 1979.

\bibitem[Wag79]{wagschal-goursat}
Claude Wagschal.
\newblock Le probl\`eme de {G}oursat non lin\'eaire.
\newblock {\em J. Math. Pures Appl. (9)}, 58(3):309--337, 1979.

\bibitem[Wag99]{wagschal-diff}
Claude Wagschal.
\newblock {\em D\'erivation, int\'egration}.
\newblock Collection M\'ethodes. [Methods Collection]. Hermann, Paris, 1999.

\bibitem[Was87]{wasow}
Wolfgang Wasow.
\newblock {\em Asymptotic expansions for ordinary differential equations}.
\newblock Dover Publications Inc., New York, 1987.
\newblock Reprint of the 1976 edition.

\bibitem[Zun05]{zung-birkhoff}
Nguyen~Tien Zung.
\newblock Convergence versus integrability in {B}irkhoff normal form.
\newblock {\em Ann. of Math. (2)}, 161(1):141--156, 2005.

\end{thebibliography}
